\documentclass[a4 paper, 11pt, reqno]{amsart}
\usepackage{latexsym}
\usepackage[T1]{fontenc}
\usepackage[english]{babel}
\usepackage{amssymb,amsmath,amsthm,amsfonts}
\usepackage[latin1]{inputenc}

\usepackage{url}
\usepackage{a4wide}
\usepackage{enumerate}
\usepackage{changebar}
\usepackage{booktabs}
\usepackage{paralist}
\usepackage{comment}
\usepackage{mathrsfs}
\usepackage{graphicx}
\usepackage{color}
\usepackage{cite}
\theoremstyle{plain}
\newtheorem{theorem}{Theorem}[section]

\newtheorem{corollary}[theorem]{Corollary}
\theoremstyle{definition}

\newtheorem{remark}[theorem]{Remark}

\def\R{{\mathbb R}}

\def\N{{\mathbb N}}

\def\cL{\mathcal{L}}

\def\e{\mathrm e}
\def\d{\mathrm{d}}

\def\div{\mathrm{div\,}}
\def\p{\mathrm{p}}
\def\q{\mathrm{q}}
\numberwithin{equation}{section}

\title[The Oseen-Navier-Stokes flow with rotating effects ]{A non-autonomous model problem for the Oseen-Navier-Stokes flow with rotating effects
} 
\author{Matthias Geissert}
\address{ Department of Mathematics\\ Technische Universit\"at Darmstadt\\ Schlossgartenstr. 7\\ 64289 Darmstadt,
Germany} \email{geissert@mathematik.tu-darmstadt.de}
\author{Tobias Hansel}
\address{ International Research Training Group 1529\\ Technische Universit\"at Darmstadt\\ Schlossgartenstr. 7 \\ 64289 Darmstadt,
Germany} \email{hansel@mathematik.tu-darmstadt.de}
\keywords{Navier-Stokes flow, Oseen flow, rotating obstacle, non-autonomous PDE} \subjclass[2000]{Primary 35Q30; Secondary  76D03, 76D05}
\thanks{The second author was supported by the DFG International Research Training Group 1529 on \emph{Mathematical Fluid Dynamics} at TU Darmstadt}
\begin{document}

%%%%%%%%%%%%%%%%%%%%%%%%%%%%%%%%%%%%%%%%%%%%%%%%%%%%%%%%%%%%%%%%%%%%%%%%%%%%%%%
%
%   Startangaben
%   title
%
%%%%%%%%%%%%%%%%%%%%%%%%%%%%%%%%%%%%%%%%%%%%%%%%%%%%%%%%%%%%%%%%%%%%%%%%%%%%%%%

\maketitle %\thispagestyle{fancy}
%\tableofcontents
%\today

%%%%%%%%%%%%%%%%%%%%%%%%%%%%%%%%%%%%%%%%%%%%%%%%%%%%%%%%%%%%%%%%%%%%%%%%%%%%%%%
%
%   Hauptteil
%
%%%%%%%%%%%%%%%%%%%%%%%%%%%%%%%%%%%%%%%%%%%%%%%%%%%%%%%%%%%%%%%%%%%%%%%%%%%%%%%
%
\begin{abstract}
Consider the Navier-Stokes flow past a rotating obstacle with
a general time-dependent angular velocity and a time-dependent
outflow condition at infinity. After rewriting the problem on a fixed domain,
one obtains a non-autonomous system of equations with unbounded drift terms. 
It is shown that the solution to a model problem in the whole space case $\R^d$
is governed by a strongly continuous evolution
system on $L^p_\sigma(\R^d)$ for $1<p<\infty$. The strategy is to
derive a representation formula, similar to the one known in the case
of non-autonomous Ornstein-Uhlenbeck equations. This explicit formula allows to prove $L^p$-$L^q$ estimates
and gradient estimates for the evolution system. These results are key ingredients to 
obtain (local) mild solutions to the full nonlinear problem by a version of Kato's iteration 
scheme. 

\end{abstract}

\section{Introduction and main result}
In this paper we consider  a model problem in $\R^d$ for the flow of an incompressible, viscous fluid
past a rotating obstacle with an additional time-dependent outflow condition at
infinity. The equations describing this problem are the Navier-Stokes
equations in an exterior domain varying in time with an additional
condition for the velocity field at infinity.

In order to motivate our model problem, let $\mathcal O \subset \R^d$ be a compact obstacle with
smooth boundary, let $\Omega:=\R^d \setminus \mathcal O$ be the
exterior of the obstacle and 
%. We are interested in the case where the
%obstacle undergoes a prescribed motion, particularly a rotation. 
%So we 
let $m \in C([0,\infty); \R^{d\times d})$ be a
continuous matrix-valued function.
%, i.e. $M(t)=-M(t)^{\mathrm T}$,
%and $M(t), M(s)$ commute\footnote{This condition can physically be interpreted by the fact that the axis of
%rotation is fixed.} for all $t,s >0$. 
Then, the exterior of the rotated obstacle at time $t>0$ is represented by $\Omega(t):=
Q(t)\Omega$ where $Q(t)$ solves the ordinary differential equation
\begin{equation}
\left\{\begin{array}{rcll}
\partial_t Q(t) &=& m(t)Q(t),& t>0,\\[0.2cm]
Q(0)&=&\mathrm{Id}.
\end{array}\right.
\end{equation}\normalsize
With a prescribed velocity field $v_\infty\in C^1([0,\infty);\R^d)$
at infinity, the equations for the fluid on the time-dependent
domain $\Omega(t)$ with no-slip boundary condition
take the form 
\begin{align}\label{eq:NS_2}
v_t-\Delta v +v\cdot \nabla v+\nabla \q&=0&\quad\quad\quad\mbox{in $\Omega(t)\times (0,\infty) $,}\notag\\
\div v&=0&\quad\quad\quad\mbox{in $\Omega(t)\times (0,\infty) $,}\notag\\
v(t,y)&=m(t)y&\quad\quad\quad\mbox{ on $\partial\Omega(t)\times (0,\infty) $,}\\
\lim_{|y|\to \infty} v(t,y)&=v_\infty(t)&\quad\quad\quad\mbox{ \mbox{for} $t\in(0,\infty) $,}\notag\\
v(0,y)&=u_0(y)&\quad\quad\quad\mbox{in $\Omega$}\notag,
\end{align}
where $v$ and $\q$ are the unknown velocity field and the pressure of
the fluid, respectively. 

The disadvantage of this description is the
variability of the domain $\Omega(t)$, and the fact that the
equations do not fit into the $L^p$-setting, due the velocity
condition at infinity. Assume for the time beeing that $m(t)$ is skew
symmetric for $t>0$; this implies that for all $t>0$ the matrix $Q(t)$ is orthogonal. Then, by setting%\vspace{-0.2cm}
\begin{equation}
x=Q(t)^{\mathrm T}y, \quad u(t,x)=Q(t)^{\mathrm T}
(v(t,y)-v_\infty(t)), \quad \p(t,x)=\q(t,y),
\end{equation}
the above equations can be transformed to the reference domain
$\Omega$ and the new velocity field $u$ vanishes at infinity. Then
\eqref{eq:NS_2} is equivalent to the following system of equations
%\small
\begin{align}\label{eq:NS_3} 
	\left.
	\begin{array}{l}
u_t- \Delta u - \mathcal M(t)x \cdot \nabla u + \mathcal M(t)u  \\[0.1cm]
\quad+ Q(t)^{\mathrm T}v_{\infty}(t)\cdot \nabla u 
-Q(t)^{\mathrm T}\partial_t v_\infty(t)\\[0.1cm]
\quad +u\cdot \nabla u
+\nabla \p
\end{array}\right\}&=0&\;\mbox{in   $\Omega \times (0,\infty)$,}\notag\\
\div u&=0&\quad\mbox{in $ \Omega \times (0,\infty) $,}\notag\\
u(t,x)&= \mathcal M(t)x-Q(t)^{\mathrm T}v_\infty(t)&\;\mbox{on $\partial\Omega \times (0,\infty)$},\\
\lim_{|x|\to \infty} u(t,x)&= 0&\;\mbox{\mbox{for} $t\in(0,\infty) $,}\notag\\
u(0,x)&=u_0(x)&\;\mbox{in $\Omega$},\notag
\end{align}%\vspace{0.3cm}
\normalsize
where $\mathcal M(t):=Q(t)^{\mathrm T}m(t)Q(t)$. The  main difficulty in dealing with this problem arises since the
term $\mathcal M(t)x \cdot \nabla$ has unbounded coefficients. In particular, the
lower order terms cannot be treated by classical perturbation theory for the
Stokes operator. 

Note that even if we assume that $m(t)\equiv m$ is independent of
time (this implies that also $\mathcal M(t)\equiv \mathcal 
M$ is independent of time), equation (\ref{eq:NS_3}) is still non-autonomous due to the
time-dependent first order term $Q(t)^{\mathrm T}v_{\infty}\cdot
\nabla$ (except in some special cases discussed below).

However, by using localization techniques similar to \cite{GHH06}, this problem is
finally reduced to a model problem in $\R^d$ and a model problem in a bounded domain. 
Since $Q(t)\partial_t v_\infty(t)\equiv F(t)$, $t>0$, i.e. it is
constant in space, we may put this term in the pressure $\p$.
Hence, in this paper we discuss the following linearized model problem
in $\R^d$\vspace{0.2cm}
\begin{align} 
	u_t-\Delta
u-\left(M(t)x + f(t)\right)\cdot\nabla u +M(t)u + \nabla \p &=0&\quad\mbox{in $\R^d \times (0,\infty)$,}\notag\\
\div u&=0& \quad\mbox{in $\R^d \times (0,\infty)$,}\label{eq:NS_1}\\
u(0)&=u_0&\quad\mbox{in $\R^d$,}\notag\vspace{0.2cm} 
\end{align}
where we allow general coefficients $M\in C([0,\infty); \R^{d\times d})$ and $f\in C([0,\infty); \R^{d})$. If we set $M(t):=Q(t)^{\mathrm T}m(t)Q(t)$ and
$f(t):=-Q(t)^{\mathrm T}v_\infty(t)$ then we obtain the linearization of equation \eqref{eq:NS_3} with $\Omega=\R^d$. 
Such a model problem also arises in the analysis of a rotating body with translational velocity
$-v_\infty(t)$, see \cite{Far05}.

Existence and uniqueness of a mild solution of an autonomous variant of problem \eqref{eq:NS_2}
{\em without} an outflow condition, i.e.  $v_\infty\equiv0$,
and $m(t)\equiv m$, was investigated in quite a few papers, see
\cite{His99a}, \cite{His99b}, \cite{GHH06} and \cite{HS05}. Hishida was even able to deal
with a time dependent rotation in \cite{His01}, however only for angular velocities of a special form.

For the problem including an additional outflow condition at infinity,
there are only a few results. 
Indeed, in the special case, where $m(t)x=\omega(t) \times x$ and
$\omega:[0,\infty) \rightarrow \R^3$ is the angular velocity of the
obstacle and $v_\infty:[0,\infty)\to\R^3$ a time-dependent outflow velocity, 
Borchers \cite{Bor92} constructed weak non-stationary solutions for the equations (\ref{eq:NS_3}). 
Moreover, Shibata \cite{Shi08} studied
the special case where $m(t)\equiv m$, $v_\infty(t)=v_\infty$ and
$mv_\infty=0$. The condition $mv_\infty=0$,  i.e.
$Q(t)^{\mathrm T}v_\infty= k v_\infty$ for $k\in \{-1,1\}$, ensures that
\eqref{eq:NS_3} is still an autonomous equation and the
solution of \eqref{eq:NS_3} is governed by a $C_0$-semigroup which is
{\em not} analytic.
The physical meaning of the additional condition $mv_\infty=0$ is that the outflow
direction of the fluid is parallel to the axis of rotation of the
obstacle. The stationary problem of this latter situation was
analysed in \cite{Far05}. 

The assumption $mv_\infty=0$ was recentely relaxed by the second author in
\cite{Han10}. Indeed, he was able to deal with the model problem in $\R^d$ where
$m(t)v_\infty\neq 0$ and $v_\infty(t)\equiv v_\infty$.
However he assumes that $m(t)$ and $m(s)$ commute for all
$t,s>0$ which can physically be interpreted by the fact that
the axis of rotation is fixed.

The aim of this work is to remove the latter additional condition, i.e.
$m(t)$ and $m(s)$ need not to commute and $v_\infty$ may be time-dependent.

As usual the Helmholtz projection $\mathbb P$ allows us to rewrite \eqref{eq:NS_1} as an abstract Cauchy problem in
$L^p_\sigma(\R^d)$,
where $L^p_\sigma(\R^d)$ denotes  the space of all
solenoidal vector fields in $L^p(\R)^d$:
\begin{equation}\label{eq:NS_5}
\begin{array}{rclll}
u'(t)- A(t)u(t) &=&0, & t>0,  \\[0.15cm]
u(0)&=&u_0.&
\end{array}
\end{equation}
Here:
\begin{align*}
	A(t)u&:=\mathbb P\left(\Delta u+\left(M(t)x + f(t)\right)\cdot\nabla u
	+M(t)u\right)\\
	D(A(t))&:=\{ u\in W^{2,p}(\R^d)^d\cap L^p_\sigma(\R^d): M(t)x\cdot
	\nabla u\in L^p(\R^d)^d\}.
\end{align*}
Note that it immediately follows from \cite{HS05} that for fixed $t>0$, the operator $A(t)$ is the
generater of a $C_0$-semigroup, which is not analytic.
The fact that the semigroup is not analytic prevents us from employing
standard generation results for evolution systems, see \cite[Chapter
5]{Paz83} and references therein. For the same reason, $L^p$-$L^q$
estimates and gradient estimates don't follow from standard arguments.

Therefore, we first derive a representation formula for the solution of
\eqref{eq:NS_1}. In order to derive this representation formula
we transform  \eqref{eq:NS_1} to a non-autonomous heat equation
which can be explicitly solved, see Section~\ref{sec:cov}.
It turns out that the transformation to a non-autonomous heat equation
is crucial to deal with our problem in this generality since the
different transformation used in \cite{Han10} caused the additional
assumption that $M(t)$ and $M(s)$ commute for all $t,s>0$.

In the following we denote by $\{U(t,s)\}_{t,s\geq 0}$ the evolution system on $\R^d$ generated by the family of matrices $\{-M(t)\}_{t\geq 0}$, i.e.
\begin{equation}
\left\{\begin{array}{lcl}
\partial_t U(t,s) & = & -M(t)U(t,s),\\[0.15cm]
U(s,s) & = & \mathrm{Id}.
\end{array}\right.
\vspace{0.1cm}
\label{eq:evol}
\end{equation}
Note that $\partial_s U(t,s)  =  U(t,s)M(s)$.

We are now ready to present our main result.

\begin{theorem}\label{thm:main}
	Let $1<p<\infty$, $M\in C([0,\infty);\R^{d\times d})$ and $f\in
	C([0,\infty);\R^d)$. The the solution of
\eqref{eq:NS_5} is governed by a strongly continuous evolution system
$\{T(t,s)\}_{t\geq s\geq0}\subset \cL(L^p_\sigma(\R^d)^d)$. 
Moreover, the evolution system $\{T(t,s)\}_{t\geq s\geq0}$ admits the
following properties:
\begin{enumerate}
	\item[(a)] For $T_0>0$ set
	$
		M_{T_0}:=\sup\{\|U(t,s)\|:t,s\in[0,T_0]\}.
	$
	Then for $1<p<\infty$ and $p\leq q \leq \infty$ there exists
	$C:=C(M_{T_0},d)>0$ such that for $u\in L_\sigma^p(\R^d)$
\begin{align}
	\|T(t,s)u\|_{L^q_\sigma(\R^d)} 
	&\leq C
(t-s)^{-\frac{d}{2}\left(\frac{1}{p}-\frac{1}{q}\right)}\|u\|_{L^p_\sigma(\R^d)},&
 0\leq s<t<T_0,\label{eq:LpLqsmoothing}\\
\|\nabla T(t,s)u\|_{L^q(\R^d)} 
&\leq C
(t-s)^{-\frac{d}{2}\left(\frac{1}{p}-\frac{1}{q}\right)-\frac{1}{2}}\|u\|_{L^p_\sigma(\R^d)},
&  0\leq s<t<T_0. \label{eq:GradientEstimate}
\end{align}
In particular, if the evolution system $\{U(t,s)\}_{s,t\geq 0}$ is uniformly
bounded, i.e. $M_{T_0}\leq M$, for some $M>0$ and all $T_0>0$, we may set
$T_0=\infty$. 
\item[(b)]
For $1<p<q<\infty$, $s\geq0$ and $u\in L^p_{\sigma}(\R^d)$ we have
\begin{align*}\label{eq:behavior_ts_1}
	\lim\limits_{t\to s,\ t> s}(t-s)^{\frac d 2 \left(\frac 1 p - \frac 1 q\right)} \|T(t,s)u\|_{L^q_{\sigma}(\R^d)}
	=0 \mbox{ and }
	\lim\limits_{t\to s,\ t> s}
(t-s)^{\frac 1 2 } \|\nabla T(t,s)u\|_{L^p(\R^d)}=0.
\end{align*}
\end{enumerate}
\end{theorem}
Next we consider the nonlinear problem
\begin{equation}\label{eq:NS_abstract}
\begin{array}{rclll}
u'(t)- A(t)u(t) + \mathbb{P}( (u(t) \cdot \nabla) u(t) )&=&0, & t>0,  \\[0.15cm]
u(0)&=&u_0,&
\end{array}
\end{equation}
with initial value $u_0 \in L^p_{\sigma}(\R^d)$.

For given $0<T_0\leq \infty$,
we call a function $u\in C([0,T_0);L^p_{\sigma}(\R^d))$ a
\textit{mild solution} of (\ref{eq:NS_abstract})  if $u$ satisfies
the integral equation
\begin{equation}\label{eq:Duhamel}
u(t) = T(t,0)u_0 - \int_0^t T(t,s) \mathbb P((u(s) \cdot \nabla)
u(s)) \d s, \quad t>0,
\end{equation}
in $L^p_{\sigma}(\R^d)$. 
By adjusting Kato's iteration scheme (see \cite{Kat84}) to our
situation the existence of a unique (local) mild solution follows, cf.
\cite{Han10} for details.
\begin{corollary}\label{cor:Kato}
Let $2\leq d\leq p \leq q < \infty$, $M\in C([0,\infty);\R^{d\times d})$, $f\in
	C([0,\infty);\R^d)$ and $u_0\in
L^p_{\sigma}(\R^d)$. Then there exists $T_0>0$ and a unique mild
solution $u\in C([0,T_0);L^p_{\sigma}(\R^d))$ of
(\ref{eq:NS_abstract}), which has the properties
\begin{equation}\label{eq:mild_solution_prop1}
t^{\frac{d}{2}\left(\frac{1}{p}-\frac{1}{q}\right)}u(t) \in
C([0,T_0);L^q_{\sigma}(\R^d)),
\end{equation}
\begin{equation}\label{eq:mild_solution_prop2}
t^{\frac{d}{2}\left(\frac{1}{p}-\frac{1}{q}\right)+\frac{1}{2}}\nabla
u(t) \in C([0,T_0);L^q(\R^d)^{d\times d}).
\end{equation}
If $p<q$, then in addition
\begin{equation}\label{eq:mild_solution_prop3}
t^{\frac{d}{2}\left(\frac{1}{p}-\frac{1}{q}\right)}\|u(t)\|_{L^q(\R^d)} +
t^{\frac{1}{2}}\|\nabla u(t)\|_{L^p(\R^d)} \rightarrow 0 \qquad \mbox{as\;} t
\rightarrow 0.
\end{equation}%\vspace{0.05cm}
Moreover, in the case $d=p$ we may set $T_0 =+\infty$
provided
$\|u_0\|_{L^d(\R^d)}$ is small enough and $\{U(t,s)\}_{s,t\geq 0}$ is
uniformly
bounded.

\end{corollary}
\begin{remark}
	In particular, $\{U(t,s)\}_{s,t\geq 0}$ is
	uniformly bounded if $M(t)$ is skew symmetric for all $t>0$.
\end{remark}
\section{Proof of Theorem~\ref{thm:main}}
Let $M$ be as in  Theorem~\ref{thm:main},
and let $\{U(t,s)\}_{s,t\geq0}$  be the evolution system on $\R^d$ that satisfies \eqref{eq:evol}. 
We consider the system of parabolic equations of the form
\begin{equation}\label{eq:1}
\left\{\begin{array}{rcll}
\partial_t u(t,x)  -\mathcal A(t) u(t,x)&=&0, & \; t> s,\  x\in \R^d,\\[0.15cm]
u(s,x) &=& \varphi(x),& \;x\in \R^d,
\end{array}\right.\end{equation}
for $s\geq 0$ fixed, initial value $\varphi\in L^p(\R^d)^d$ and some $p\in(1,\infty)$. Here the family of operators $\mathcal A(t)$ is of the form
$$
\mathcal A(t)u(x) := \Big( \Delta u_i(t,x) + \langle M(t)x + f(t) , \nabla
u_i(t,x) \rangle \Big)_{i=1}^{d} - M(t)u(t,x), \quad t>0,\; x\in\R^d.
$$
 As in \cite[Lemma
3.2]{GL08a} or \cite{Han10}, we first develop an explicit representation
formula. To be more precise,
 we show in Section~\ref{sec:cov} that for $p\in(1,\infty)$ and $\varphi\in L^p(\R^d)^d$
 the solution $u$ to (\ref{eq:1}) is governed by a strongly continuous evolution system
 $\{\tilde T(t,s)\}_{t\geq s}\subset \cL(L^p(\R^d)^d)$ which is explicitly given by 
 \begin{align}
u(t,x):=& (\tilde T(t,s)\varphi)(x):=(k(t,s,\cdot)*\varphi)(U(s,t)x+g(t,s)),
%=&\frac{1}{(4\pi)^{d/2} (\det Q_{t,s})^{1/2}} U(t,s) \int_{\R^d}
%\varphi(U(s,t)x+g(t,s)-y)\mathrm e^{-\frac 1 4 \langle
%Q_{t,s}^{-1}y,y\rangle}  \d y, 
\quad t>s,\ x\in\R^d,
\label{eq:solformul}
\end{align}
where
 \begin{align}
	k(t,s,x)&:= \frac{1}{(4\pi)^{d/2} (\det Q_{t,s})^{1/2}} U(t,s) 
\mathrm e^{-\frac 1 4 \langle
Q_{t,s}^{-1}x,x\rangle}  \d y,\quad t>s\geq 0,\ x\in\R^d,\label{eq:kern}
\\
g(t,s)&:= \int_s^t U(s,r) f(r) \d r,\quad
Q_{t,s}: = \int_s^t U(s,r)U^*(s,r) \d r, \quad t\geq s\geq0.
\nonumber
\end{align}
%
%we show in Section~\ref{sec:proof} that for $\varphi\in
Similar to \cite{DPL07} one can show that 
for $\varphi\in
C_c^\infty(\R^d)^d$ the solution $u$ of \eqref{eq:1} given by \eqref{eq:solformul}  is a classical
solution. 

A simple calculation shows that $\div \tilde T(t,s)\varphi=0$ for $\varphi\in
C_{c,\sigma}^\infty(\R^d)$ and $t\geq s\geq0$. Hence, the restriction
$T(t,s):=\tilde T(t,s)|_{L^p_\sigma(\R^d)}$ is an evolution system on
$L^p_\sigma(\R^d)$. In particular, $u(t):=T(t,0)u_0$ is a solution to
\eqref{eq:NS_5}.

%Once the new representation formula \eqref{eq:solformul} is established
%we may closely follow the lines of the proofs given in \cite{Han09}.
%However, for the convenience of the reader we sketch the basic ideas of
%the proofs.

By similar arguments as in the proofs of \cite[Lemma
3.2]{GL08a} or \cite[Lemma 2.4]{Han10}, for $T_0>0$ there exists $C:=C(d,M_{T_0})>0$ (see
Theorem~\ref{thm:main} for the definition of $M_{T_0}$)
such that
\begin{align}
	\begin{split}
\|Q_{t,s}^{-\frac 1 2}\| \leq C (t-s)^{-\frac 1 2}, \quad 0\leq s<t<T_0,\\
(\det Q_{t,s})^{\frac 1 2} \geq
C(t-s)^{\frac d 2},  \quad 0\leq s<t<T_0.
\end{split}
\label{eq:estimateQ}
\end{align}
Moreover, if $M_{T_0}$ is uniformly bounded in $T_0$ we may write
$T_0=\infty$ in \eqref{eq:estimateQ}.

\begin{proof}[Proof of Theorem~\ref{thm:main}]
We start by showing the estimate (\ref{eq:LpLqsmoothing}). Let $T_0>0$.
By the change
of variables $\xi=U(s,t)x$ and by Young's inequality we obtain
\begin{align*}
 \|T(t,s)u\|_{L_{\sigma}^q(\R^d)}  
 \leq |\det U(s,t)|^{\frac1q}\|k(t,s,\cdot)\|_{L^r(\R^d)}
\|u\|_{L_{\sigma}^p(\R^d)},\quad t> s\geq0,
% \leq \frac{\|U(t,s)\|}{(4\pi)^{\frac d 2} (\det
%Q_{t,s})^{1/2}} \, |\det
%U(s,t)|^{1/p}\left(\int_{\R^d}\left|\e^{-\frac1 4 \langle
%Q_{t,s}^{-1}y,y \rangle}\right|^r \d y\right)^{1/r}
%\|u\|_{L_{\sigma}^p(\R^d)},
\end{align*}
where $1<r<\infty$ with $\frac 1 p + \frac 1 r = 1 + \frac 1 q$.
Further, by the change of variable $y= Q_{t,s}^{1/2} z$ we obtain
\begin{align*}
	\|k(t,s,\cdot)\|_{L^r(\R^d)}^r
%	&=\int\limits_{\R^d}|k(t,s,y)|^r\d y
%	=\int\limits_{\R^d}	\left|\frac1{(4\pi)^\frac d2}\e^{-\frac1 4 \langle Q_{t,s}^{-1}y,y
%\rangle} \right|^r \d y
&= \|U(t,s)\|\int_{\R^d}
\left(\frac1{(4\pi)^\frac d2}
\e^{-\frac{ |z|^2}{4}}\right)^r (\det Q_{t,s})^{\frac {1-r}2} \d
z\\
&\leq C \|U(t,s)\|(\det Q_{t,s})^{\frac{1-r}2},\quad t\geq s\geq0, 
%\int_{\R^d}\left|\e^{-\frac1 4 \langle Q_{t,s}^{-1}y,y
%\rangle} \right|^r \d y= \int_{\R^d}
%\e^{-\frac{r |z|^2}{4}} (\det Q_{t,s})^{1/2} \d
%z \leq C (\det Q_{t,s})^{1/2}, 
\end{align*}
for some $C>0$. Now \eqref{eq:estimateQ} 
yields \eqref{eq:LpLqsmoothing}.

To prove the gradient estimate (\ref{eq:GradientEstimate}), we first
observe that 
%Vit follows from 
%\begin{align*}
%	\nabla
%	k(t,s,y)=k(t,s,y)\cdot\left(U^T(s,t)Q^{-1}_{s,t}y\right)^T,\quad t\geq
%	s\geq0,\ y\in\R^d
%\end{align*}
%that
\begin{align*}
	\nabla
	T(t,s)u(x)=\int\limits_{\R^d}u(U(s,t)x+g(t,s)k(t,s,y)\left(U^T(s,t)Q^{-1}_{s,t}y\right)^T\d
	y,\quad t> s\geq0,\ x\in\R^d.
\end{align*}
Now, \eqref{eq:GradientEstimate} follows similarly as above.

Since \eqref{eq:1} is uniquely solvable for $\varphi\in C_c^\infty(\R^d)^d$, see Section~\ref{sec:cov}, the
law of evolution is valid, i.e.
\begin{equation}\label{eq:law of evolution}
\tilde T(t,s)\varphi = \tilde T(t,r)\tilde T(r,s)\varphi,
\end{equation}
holds for $ 0\leq s \leq r \leq t $ and every $\varphi \in C_{c}^\infty(\R^d)^d$. 
The density of $C_{c}^\infty(\R^d)^d$ in $L^p(\R^d)^d$ yields that (\ref{eq:law of evolution}) even holds for all $\varphi\in L^p(\R^d)^d$. 

In order to prove the strong continuity of the map $(t,s)\mapsto \tilde T (t,s)$ on $0\leq s \leq t $ we apply the change of the variables $y=Q_{t,s}^{1/2}z$, to see that
$$
\tilde T(t,s)\varphi(x)= \frac{1}{(4\pi)^{\frac d 2}}U(t,s)\cdot \int_{\R^d}
\varphi(U(s,t)x+g(t,s)-Q_{t,s}^{\frac 1 2}z) \e^{-\frac{|z|^2}{4}} \d z
$$
holds. For $t>s$ fixed, we pick two sequences $(t_n)_{n\in \N}$ and $(s_n)_{n\in \N}$ such that $t_n \geq s_n $ holds for every $n\in \N$ and $(t_n,s_n)\to (t,s)$ as $n \to \infty$. For every $\varphi \in C_c^{\infty}(\R^d)^d$ and every
$x\in\R^d$ we now obtain
$$
\varphi(U(s_n,t_n)x+g(t_n,s_n)-Q_{t_n,s_n}^{\frac 1 2}z)
\rightarrow \varphi(U(s,t)x+g(t,s)-Q_{t,s}^{\frac 1 2}z)
$$
as $n\rightarrow \infty$. Lebegue's theorem now yields
$\tilde T(t_n,s_n)\varphi \rightarrow\tilde T(t,s)\varphi$
as $n\rightarrow \infty$ for every $\varphi\in
C_c^{\infty}(\R^d)^d$. The density of $C_c^{\infty}(\R^d)^d$ in
$L^p(\R^d)^d$ implies the strong continuity.

In order to prove Theorem~\ref{thm:main}(b)
let $u\in L^p_\sigma(\R^d)$, $t-s \leq 1$ and choose $(u_n)_{n\in\N}
\subset C_{c,\sigma}^{\infty}(\R^d) \subset L_{\sigma}^p(\R^d)$,
such that $\lim_{n\to\infty}\|u-u_n\|_{L^p(\R^d)}=0$. The triangle inequality
together with the $L^p$-$L^q$ estimates (\ref{eq:LpLqsmoothing})
imply that there exist constants $C_1, C_2>0$ such that 
\begin{align*}
&(t-s)^{\frac d 2 \left(\frac 1 p - \frac 1 q \right)}
\|T(t,s)u\|_{L_{\sigma}^q(\R^d)}\\
&\qquad \leq (t-s)^{\frac d 2 \left(\frac 1 p - \frac 1 q \right)}
\|T(t,s)(u-u_n)\|_{L_{\sigma}^q(\R^d)} + (t-s)^{\frac d 2 \left(\frac 1 p - \frac 1
q
\right)} \|T(t,s)u_n\|_{L_{\sigma}^q(\R^d)}\\
&\qquad \leq  C_1 \|u-u_n\|_{L_{\sigma}^p(\R^d)} + C_2 (t-s)^{\frac d 2 \left(\frac 1 p -
\frac 1 q \right)} \|u_n\|_{L_{\sigma}^q(\R^d)},\quad 0\leq t-s\leq 1,\
n\in\N.
\end{align*}
Hence,
$
	\lim\limits_{t\to s}(t-s)^{\frac d 2 \left(\frac 1 p - \frac 1 q
	\right)} \|T(t,s)u\|_{L_{\sigma}^q(\R^d)}=0
$
by letting first $t\rightarrow s$ and then $n\rightarrow \infty$.
The second assertion in Theorem~\ref{thm:main}(b) is proved in a similar
way.
\end{proof}

\section{Representation Formula}
\label{sec:cov}
In this section the representation formula (\ref{eq:solformul}) is derived. The general idea is to do a coordinate transformation in order to eliminate the unbounded drift and the zero order term of the operator $\mathcal A(t)$. For this purpose we set
\begin{equation*}
z := U(s,t)x +g(t,s), 
\end{equation*}
where
\begin{equation*}
g(t,s):= \int_s^t U(s,r) f(r) \d r,
\end{equation*}
and we look for a solution $u$ of (\ref{eq:1}) with initial value  $\varphi\in C_c^{\infty}(\R^d)^d$ in the form
\begin{equation}\label{eq:coord_transf}
u(t,x)= U(t,s) w(t,U(s,t)x+g(t,s)).
\end{equation}
%Since we now work in the whole space $\R^d$ our domain is invariant under this transformation and thus we do not have to pay the prize of obtaining a time-dependent domain.

%
By recalling \eqref{eq:evol} we obtain from a straightforward computation that
\begin{align*}
\partial_t u(t,x) &= -M(t) U(t,s)w(t,z) + U(t,s) \Big(\langle U(s,t) M(t) x + U(s,t) f(t), \nabla w_i(t,z) \rangle\Big)_{i=1}^d \\
&\qquad \qquad + U(t,s)\partial_t w (t,z),
\end{align*}\normalsize
holds. Moreover, we can write equation \eqref{eq:coord_transf} component-wise as
\begin{equation*}
u_i(t,x)= \sum_{j=1}^d U_{ij}(t,s) w_j(t,U(s,t)x+g(t,s)), \qquad \mbox{for}\; i=1,\ldots,d,
\end{equation*}
and thus for the spatial derivatives of $u$ we obtain
\begin{align*}
\nabla u_i(t,x) &= \sum_{j=1}^d U_{ij}(t,s) U^*(s,t) \nabla w_j(t,z),\\[0.1cm]
\nabla^2 u_i(t,x) &= \sum_{j=1}^d U_{ij}(t,s) U^*(s,t) \nabla^2 w_j(t,z)U(s,t).\vspace{0.2cm}
\end{align*}
In particular, the drift term can be written as
\begin{align*}
\langle M(t)x + f(t), \nabla u_i(t,x) \rangle &= %\sum_{j=1}^d U_{ij}(t,s) \langle M(t)x  + f(t), U^*(s,t) D_z w_j(t,z) \rangle \\
  \sum_{j=1}^d U_{ij}(t,s) \langle U(s,t) M(t)x + U(s,t)f(t) , \nabla w_j(t,z) \rangle. 
\end{align*}\normalsize
Thus, the function $u$ solves problem (\ref{eq:1}) if and only if for every $i=1,\ldots,d$, the function $w_i : \R^d \rightarrow \R$ is a solution to\vspace{0.1cm}
\begin{equation}\label{eq:2}
\left\{\begin{array}{lcll}
\partial_t w_i(t,z)  &=&\mathrm{Tr} [U(s,t) U^*(s,t) \nabla^2 w_i(t,z)], &  \; t> s, z\in \R^d,\\[0.15cm]
w_i(s,z) &=& \varphi_i(z),&  \;z\in \R^d.
\end{array}\right.\vspace{0.1cm}\end{equation}\normalsize
By our transformation we now obtained an uncoupled system of parabolic equations with coefficients only depending on $t$. More 
precisely, for $i=1,\ldots,d$, the equation \eqref{eq:2} is a non-autonomous heat equation. It is well known that such a problem can be 
uniquely solved (cf. \cite[Proposition 2.1]{DPL07}) and that for every $\varphi_i \in C_{c}^\infty(\R^d)$ its unique solution is explicitly given
by the formula
\begin{equation}
w_i(t,z) = \frac{1}{(4\pi)^{\frac d 2} (\det Q_{t,s})^{\frac 1 2}} \int_{\R^d} \varphi_i(z-y)\mathrm e^{-\frac 1 4 \langle Q_{t,s}^{-1}y,y\rangle}  \d y,
\end{equation}
where 
\begin{equation}
Q_{t,s} = \int_s^t U(s,r)U^*(s,r) \d r.
\end{equation}
Now, via \eqref{eq:coord_transf}, the unique solution to our original problem (\ref{eq:1}) is given by the representation formula
 \begin{equation}\label{eq:OUsol}
u(t,x)= (k(t,s,\cdot)\ast u)(U(s,t)x+g(t,s)),
\end{equation}
where the kernel $k(t,s,x)$ is defined in \eqref{eq:kern}.

Note that the right hand side of (\ref{eq:OUsol}) is even well defined
for each $L^p(\R^d)^d$-function $\varphi$. Thus, this explicit formula can be used to define an evolution system on $L^p(\R^d)^d$ in the following way. For $\varphi \in L^p(\R^d)^d$ we set 
\begin{equation*}
\tilde T(t,s)\varphi := \left\{\begin{array}{cl}
\varphi & \mbox{for}\; t=s,\\[0.2cm]
(k(t,s,x)\ast \varphi)(U(s,t)x+g(t,s))&\mbox{for}\; t>s.
\end{array}\right.
\end{equation*}

%From \eqref{} it immediately follows that $\{\tilde T(t,s)\}_{t\geq s \geq 0}$ is a two-parameter family of linear, bounded operators in $L^p(\R^d)^d$. 

Since problem (\ref{eq:2})  is uniquely solvable it follows via (\ref{eq:coord_transf}) that $\tilde T(t,s)\varphi$ is the unique solution of (\ref{eq:1}) for initial value $\varphi \in C_{c}^\infty(\R^d)^d$.

\bibliographystyle{amsalpha-ag4}
\bibliography{ref}

\def\cprime{$'$}
\providecommand{\bysame}{\leavevmode\hbox to3em{\hrulefill}\thinspace}
\providecommand{\MR}{\relax\ifhmode\unskip\space\fi MR }
% \MRhref is called by the amsart/book/proc definition of \MR.
\providecommand{\MRhref}[2]{%
  \href{http://www.ams.org/mathscinet-getitem?mr=#1}{#2}
}
\providecommand{\href}[2]{#2}
\begin{thebibliography}{GHH06}

\bibitem[Bor92]{Bor92}
W.~Borchers, \emph{Zur {S}tabilit\"at und {F}aktorisierungsmethode f\"ur die
  {N}avier-{S}tokes-{G}leichungen inkompressibler viskoser {F}l\"ussigkeiten},
  Habilitation, 1992, Habilitation, Universit\"{a}t Paderborn.

\bibitem[DPL07]{DPL07}
G.~Da~Prato and A.~Lunardi, \emph{Ornstein-{U}hlenbeck operators with time
  periodic coefficients}, J. Evol. Equ. \textbf{7} (2007), 587--614.

\bibitem[Far05]{Far05}
R.~Farwig, \emph{An {$L^q$}-analysis of viscous fluid flow past a rotating
  obstacle}, Tohoku Math. J. (2) \textbf{58} (2005), 129--147.

\bibitem[GHH06]{GHH06}
M.~Geissert, H.~Heck, and M.~Hieber, \emph{{$L\sp p$}-theory of the
  {N}avier-{S}tokes flow in the exterior of a moving or rotating obstacle}, J.
  Reine Angew. Math. \textbf{596} (2006), 45--62.

\bibitem[GL08]{GL08a}
M.~Geissert and A.~Lunardi, \emph{Invariant measures and maximal {$L\sp 2$}
  regularity for nonautonomous {O}rnstein-{U}hlenbeck equations}, J. Lond.
  Math. Soc. (2) \textbf{77} (2008), 719--740.

\bibitem[Han10]{Han10}
T.~Hansel, \emph{On the navier-stokes equations with rotating effect and
  prescribed outflow velocity},  {\em Journal of Mathematical Fluid Mechanics},
  to appear.

\bibitem[His99a]{His99a}
T.~Hishida, \emph{An existence theorem for the {N}avier-{S}tokes flow in the
  exterior of a rotating obstacle}, Arch. Ration. Mech. Anal. \textbf{150}
  (1999), 307--348.

\bibitem[His99b]{His99b}
T.~Hishida, \emph{The {S}tokes operator with rotation effect in exterior
  domains}, Analysis (Munich) \textbf{19} (1999), 51--67.

\bibitem[His01]{His01}
T.~Hishida, \emph{On the {N}avier-{S}tokes flow around a rigid body with a
  prescribed rotation}, Proceedings of the {T}hird {W}orld {C}ongress of
  {N}onlinear {A}nalysts, {P}art 6 ({C}atania, 2000), vol.~47, 2001,
  pp.~4217--4231.

\bibitem[HS05]{HS05}
M.~Hieber and O.~Sawada, \emph{The {N}avier-{S}tokes equations in {$\mathbb
  R\sp n$} with linearly growing initial data}, Arch. Ration. Mech. Anal.
  \textbf{175} (2005), 269--285.

\bibitem[Kat84]{Kat84}
T.~Kato, \emph{Strong {$L\sp{p}$}-solutions of the {N}avier-{S}tokes equation
  in {${\bf R}\sp{m}$}, with applications to weak solutions}, Math. Z.
  \textbf{187} (1984), 471--480.

\bibitem[Paz83]{Paz83}
A.~Pazy, \emph{Semigroups of linear operators and applications to partial
  differential equations}, Springer-Verlag, New York, 1983.

\bibitem[Shi08]{Shi08}
Y.~Shibata, \emph{On the {O}seen semigroup with rotating effect}, Functional
  analysis and evolution equations, Birkh\"auser, Basel, 2008, pp.~595--611.

\end{thebibliography}

\end{document}